\documentclass[11pt]{amsart}

\usepackage{amssymb,amsmath,amsthm,enumerate,latexsym, graphics,shapepar,nccmath,mathtools}
\usepackage{tikz-cd} 
\usepackage[all,2cell,dvips]{xy}
\usepackage[utf8]{inputenc}

\usepackage[colorlinks=true, pdfstartview=FitV, linkcolor=red, citecolor=green, urlcolor=blue]{hyperref}

\usepackage{cleveref}

\usepackage[margin=0.9in]{geometry}
\usepackage{slashbox}

\usepackage{parskip}
\newtheorem{thm}{Theorem}[section]
\newtheorem{lemma}[thm]{Lemma}
\newtheorem{prop}[thm]{Proposition}
\newtheorem{cor}[thm]{Corollary}
\newtheorem{question}[thm]{Question}

\theoremstyle{definition}

\newtheorem{rmk}[thm]{Remark}
\newtheorem{example}[thm]{Example}

\newcommand{\Hom}{\mbox{\rm Hom}}

\newcommand{\soc}{\mbox{\rm soc}}                     
\newcommand{\depth}{\mbox{\rm depth}}

\newcommand{\pdim}{\mbox{\rm pdim}}
\newcommand{\reg}{\mbox{\rm reg}}

\newcommand{\coker}{\mbox{\rm coker}}

\newcommand{\codim}{\mbox{\rm{codim}}}

\newcommand{\BB}{\mbox{$\mathbb B$}}

\newcommand{\BF}{\mbox{$\mathbb F$}}

\newcommand{\BN}{\mbox{$\mathbb N$}}

\newcommand{\BQ}{\mbox{$\mathbb Q$}}

\newcommand{\BZ}{\mbox{$\mathbb Z$}}

\newcommand{\m}{\mbox{$\mathfrak m$}}

\newcommand{\sk}{\mbox{$\mathsf{k}$}}

\begin{document}

\author[H. Ananthnarayan]{H. Ananthnarayan}
\address{Department of Mathematics, I.I.T. Bombay, Powai, Mumbai 400076.}
\email{ananth@math.iitb.ac.in}

\author[Omkar Javadekar]{Omkar Javadekar}
\address{Department of Mathematics, I.I.T. Bombay, Powai, Mumbai 400076.}
\email{omkar@math.iitb.ac.in}

\author[Rajiv Kumar]{Rajiv Kumar}
\address{Department of Mathematics, Indian Institute of Technology Jammu, J\&K, India - 181221.}
\email{rajiv.kumar@iitjammu.ac.in}

\begin{center}
\end{center}
\title{Betti Cones over Fibre Products}



\begin{abstract}
Let $R$ be a fibre product of standard graded algebras over a field. We study the structure of syzygies of finitely generated graded $R$-modules. As an application of this, we show that the existence of an $R$-module of finite regularity and infinite projective dimension forces $R$ to be Koszul. We also look at the extremal rays of the Betti cone of finitely generated graded $R$-modules, and show that when $\depth(R)=1$, they are spanned by the Betti tables of pure $R$-modules if and only if $R$ is Cohen-Macaulay with minimal multiplicity.
\end{abstract}
\maketitle
\section{Introduction}

Let $R$ be a standard graded algebra over an infinite field $\sk$, and $\mathbb B_{\mathbb Q}(R)$ denote the rational cone spanned by Betti tables of finitely generated graded $R$-modules. In \cite{ES09}, Eisenbud and Schreyer showed that if $R=\sk[X_1,\ldots,X_n]$, then the extremal rays of $\mathbb B_{\mathbb Q}(R)$ are spanned by the Betti tables of Cohen-Macaulay modules with a pure resolution. This result was conjectured by Boij and S\"oderberg in \cite{BS08}, where they also provided a proof when $n=2$. For an introduction and survey of Boij-S\"oderberg theory, we refer the reader to \cite{Fl12}.

After the work of Eisenbud and Schreyer, several authors took up the study of Betti cones and purity of their extremal rays for special classes of standard graded $\sk$-algebras.
In \cite[Theorem 4.3]{AK20}, Ananthnarayan and Kumar proved the purity of extremal rays of $\BB_{\mathbb Q}(R)$ for all standard graded $\sk$-algebras with Hilbert series $(1+nz)/(1-z)$. They also gave a complete description of the extremal rays of $\mathbb B_{\mathbb Q}(R)$. This was inspired by, and generalizes the results proved for quadratic hypersurfaces of embedding dimension two (\cite{BBEG12}), and the coordinate ring of three non-collinear points in the projective plane (\cite{GS16}). 

In general, the converse of the above result of Ananthnarayan-Kumar is false, for example it is known that the extremal rays of $\mathbb B_{\mathbb Q}(R)$ are spanned by pure modules when $R = \sk[X,Y]/\langle X^2, Y^2 \rangle$ (see \cite{Gi13}). In this article, we prove that a partial converse holds in Theorem \ref{MaintheoremFibreProducts} when $R$ is a fibre product (over the residue field). We show that if $\depth(R) = 1$ and the extremal rays of $\mathbb B_{\mathbb Q}(R)$ are pure, then $R$ is Cohen-Macaulay with minimal multiplicity, i.e., the Hilbert series of $R$ is of the form $(1+nz)/(1-z)$ for some $n \in \BN$. Using \cite[Theorem 7.3.4]{Ku17}, this can further be strengthened to say that if $R$ is a fibre product such that $\mathbb B_{\mathbb Q}(R)$ is spanned by Betti tables of pure modules, then $R$ is a level Cohen-Macaulay algebra.

One of the initial steps in understanding the Betti table of a module $M$ is to get information about its graded syzygies. Thus, we begin by studying the structure of syzygies of modules over a fibre product ring $R$. Let $(R_1,\m,\sk), (R_2,\m,\sk)$ be standard graded $\sk$-algebras, and $R= R_1 \times_{\mathsf k} R_2$ denote their fibre product. 
For each $R$-module $M$, it is known by work of Dress and Kr\"{a}mer in \cite{DS75} that the $i^{th}$ graded syzygy module, $\Omega_i^R(M)$, for $i\geq 2$, 
can be decomposed as a direct sum of an $R_1$-module and an $R_2$-module. We observe the same, and further describe the structure of $\Omega_i^R(M)$ ($i\geq 2$) in terms of first syzygy of $\pi_1(\Omega_{i-1}^R(M))$ and $\pi_2(\Omega_{i-1}^R(M))$, where $\pi_j:R\to R_j$ are the natural projections (see Proposition  \ref{first_syzygy} and Corollary \ref{cor:splitSyzygy}).

Due to the work of Avramov and Peeva \cite{AP01}, it is known that a $\sk$-algebra $R$ is Koszul if and only if every finitely generated graded $R$-module has finite regularity. If $R$ is a fibre product, we prove a stronger statement. We show in Proposition \ref{prop:koszulChar} that the Koszul property of fibre products is characterized merely by the existence of a module of infinite projective dimension and finite regularity. We also prove that the regularity of a pure $R$-module can be detected by looking at the degrees of the generators of the first two syzygy modules (see Remark \ref{regularityRemark}).

The organization of this article is as follows: In Section 2, we collect the necessary definitions, basic observations, and known results that are needed for the rest of the article. 
Section 3 is devoted to the discussion of syzygies of modules over fibre products. We start by describing the structure of syzygies of an $R_1$-module considered as an $R$-module in Proposition \ref{FieldSyzygyLemma}. As noted above, we also study the structure of $\Omega_i^R(M)$ ($i \geq 2$) of any $R$-module $M$, and show that it splits as a direct sum of an $R_1$-module and an $R_2$-module. 
These are used to study the Koszul property of fibre products in Section 4.
For example, we show that if $R$ is Koszul, then for each $R_1$-module $M$ we have $\reg_{R}(M)=\reg_{R_1}(M)$ (see Proposition \ref{Prop-equality-of-reg}). 
In Corollary \ref{UniversallyKoszulCorollary}, we also give an alternate proof for the fact that $R$ is universally Koszul if and only if the same is true for $R_1$ and $R_2$.
Finally, Section 5 focuses on the equality of the Betti cone $\mathbb B_{\mathbb Q}(R)$ and the pure Betti cone $\mathbb B_{\mathbb Q}^{\text{pure}}(R)$ over fibre products. We characterize this property when $\depth(R) = 1$ in Theorem \ref{thm:mainThmfibreProducts}. In Proposition \ref{prop:gorensteinArtinian}, we also give a class of examples of $R$, with $\depth(R) = 0$ where this property fails.

\section{Preliminaries}

In this section, we state the basic definitions and related observations needed in the rest of the article. Known results which are used in our article are also recorded here for the sake of completeness.

\subsection*{Notation:} Throughout the article, $\mathsf k$ denotes an infinite field, and $R$ is a standard graded $\sf k$-algebra, i.e., $R_0 = \sk$ and $R$ is generated as a $\sf k$-algebra by finitely many elements of degree one. We denote the unique homogeneous maximal ideal $\bigoplus_{i\geq 1} R_i$ of $R$ by $\mathfrak m$.

\subsection{Graded Rings and Modules}\label{Def:GradedModules}
Let $M =  \bigoplus_{j \in \mathbb{Z}} M_j$ and $N =  \bigoplus_{j \in \mathbb{Z}} N_j$ be a graded $R$-module.
\begin{enumerate}[{\rm a)}]
\item The \emph{$n$-twist} of a graded module $M$, denoted by $M(n)$, is the graded module defined as \\ $M(n)_j= M_{n+j}$ for all $j \in \BZ$.

\item The socle of $R$ is defined as $\soc(R) = 0:_R\m$. If $R$ is Cohen-Macaulay, we say that $R$ is \emph{level} if there exists a maximal $R$-regular sequence of linear forms $\underline{x}$ such that $R/\langle \underline{x} \rangle$ is an Artinian ring,  whose socle lies in a single degree, i.e., there is a $d \in \mathbb N$ such that $\soc(R/\langle \underline x \rangle)$ is isomorphic to finitely many copies of $\sf k(d)$.

\item\label{HilbertSeries} If $M$ is finitely generated, the \emph{Hilbert series of $M$} is defined as $H_M(z)= \sum_{j\in \mathbb{Z}} \dim_{\mathsf k}(M_j) z^j$.\\ 
It is well-known (e.g., see \cite[Section 4.1]{BH93}) that there exists $f(z) \in \mathbb{Z}[z,z^{-1}]$ such that\\ 
$H_M(z)= f(z)/(1-z)^d$, where $d= \dim(M)$ and $f(1) \in \mathbb N$.

\item An $R$-linear map $\phi:M \to N$ called a \emph{graded map of degree $n$} if $\phi(M_j) \subset N_{n+j}$ for all $j \in \BZ$. By convention, the term `graded map' means a graded map of degree zero.

\end{enumerate} 

\subsection{Invariants in Graded Resolutions}
 Let $M$ be a finitely generated graded $R$-module and 
\[\BF_{\bullet} : \cdots \xrightarrow[]{} F_n \xrightarrow[]{\phi_n} F_{n-1} \xrightarrow[]{\phi_{n-1}} \cdots \xrightarrow[]{\phi_1}F_0 \xrightarrow[]{\phi_0} M \xrightarrow[]{} 0\]
be a minimal free resolution of $M$ over $R$, i.e., for each $i$, $\phi_i$ is a graded map of degree zero, and $\ker(\phi_i) \subset \m F_i$. 

\begin{enumerate}[{\rm a)}]
\item We say that $\phi_1$ is a \emph{minimal presentation matrix} of $M$ over $R$. 

\item The module $\Omega_i^R(M)=\ker(\phi_{i-1})$ is a graded $R$-module, called the \emph{$i^{th}$ syzygy module} of $M$. The number of minimal generators of $\Omega_i^R(M)$ in degree $j$ is denoted by $\beta_{i,j}^R(M)$, and is called the  \emph{$(i,j)^{th}$ graded Betti number} of $M$. The number $\beta_i^R(M)=  \sum_{j}^{} \beta_{i,j}^R(M)$ is called the \emph{total $i^{th}$ Betti number of $M$}, and equals the minimal number of generators of $\Omega_i^R(M)$. 

\item Let $\beta_{i,j}=\beta_{i,j}^R(M)$. Then the \emph{Betti table} of $M$ is written as 

\begin{center} $\beta^R(M) =$
\begin{tabular}{|c |c c c c c|}
\hline
\backslashbox{j}{i}  & 0 & 1 & $\cdots$ & i & $\cdots$ \\
\hline
\vdots  & \vdots & \vdots & $\cdots$ & \vdots & \vdots \\

   0      & $\beta_{0,0}$ & $\beta_{1,1}$ & $\cdots$ & $\beta_{i,i}$ & \vdots \\

1 & $\beta_{0,1}$ & $\beta_{1,2}$ & $\cdots$ & $\beta_{i,i+1}$ & \vdots \\

\vdots & \vdots & \vdots & $\cdots$ & \vdots & \vdots \\

\hline
\end{tabular} \\
\end{center} 
Note that the $(i,j)^{th} $ entry in $\beta^R(M)$ is $\beta_{i,i+j}$.
\item The series $\mathcal{P}^R_M(z)= \sum_{i\geq 0} \beta_i^R(M)z^i$ (or simply $\mathcal{P}_M(z)$) is called the \emph{Poincaré series of $M$}. 

\item 
The \emph{regularity of $M$} is defined as
\[ \reg_R(M) = \sup \{ j-i \mid \beta_{i,j}^R(M) \neq 0\}. \]
\end{enumerate}

The following result is a crucial component of Theorem \ref{MaintheoremFibreProducts}.
\begin{rmk}\label{lemma-generators-up-to-degree-j}
\cite[Lemma 4.6]{AK20} Let $M$ be a finitely generated graded $R$-module,
Let $M^{(j)}$ be the submodule of $M$ generated by elements of degree at most $j$. Then $\Omega_i^R(M)^{(j+i)} \simeq \Omega_i^R(M^{(j)})^{(j+i)}$ and $\beta_{i, k+i}^R(M) = \beta_{i, k+i}^R(M^{(j)})$ for $k \leq j$ and $i \geq 0$.
\end{rmk}

\subsection{Pure and Linear Resolutions, Koszul Algebras} Let $\BF_{\bullet}$ be a minimal graded free resolution of $M$ over $R$.
\begin{enumerate}[a)]
\item We say that $\BF_{\bullet}$ is \emph{pure} if for every $i$, $\beta_{i,j}^R(M) \neq 0$ for at most one $j$. In such a case, $M$ is said to be a \emph{pure module} of type 
\begin{enumerate}[{\rm i)}]
\item ${\mathbf{\delta}}=(\delta_0, \delta_1,\delta_2,\ldots)$ if $\pdim_R(M)= \infty$ and $\beta_{i,\delta_i}^R(M) \neq 0$ for all $i \geq 0$. 

\item $\delta= (\delta_0, \delta_1, \ldots, \delta_p, \infty, \infty,\ldots)$ if $\pdim_R(M)=p$ and $\beta_{i,\delta_i}^R(M) \neq 0$ for $0 \leq i \leq p$.

\end{enumerate}

\item A pure module $M$, generated in degree $\delta_0$, is said to have a \emph{linear resolution} if $\beta_{i,j}^R(M) \neq 0$ implies that $j = \delta_0 + i$. We say that a module $M$ generated in degree $\delta_0$ is \emph{linear up to the $k^{th}$ stage} if for $i \leq k$, $\beta_{i,j}^R(M)=0$ for $j \neq \delta_0+i$.

\item The ring $R$ is said to be a \emph{Koszul algebra} if $\sk$ has a linear resolution over $R$, and a \emph{universally Koszul algebra} if every ideal of $R$ generated by linear forms has a linear resolution.

\end{enumerate}

\subsection{Betti Cones and Extremal Rays} Let $R$ be a standard graded $\sk$-algebra.
\begin{enumerate}[a)]
\item Then the \emph{Betti cone of finitely generated $R$-modules} is
\[ \mathbb{B}_{\mathbb{Q}}(R) = \{ c_1 \beta^R(M_1) + \cdots +c_n \beta^R(M_n) \mid c_i \in \mathbb{Q}_{\geq 0}, M_i {\text{\ is a finitely generated\ }} R{\text{-module}}\}. \]
Similarly, we define the \emph{Betti cone of pure $R$-modules} as
\[ \BB^{\text{pure}}_{\mathbb{Q}}(R) = \{ c_1 \beta^R(M_1) + \cdots +c_n \beta^R(M_n) \mid c_i \in \mathbb{Q}_{\geq 0}, M_i {\text{\ is a finitely generated pure\ }} R{\text{-module}}\}. \]

\item A Betti table $\beta^R(M)$ of a nonzero finitely generated $R$-module $M$ is said to be \emph{extremal} in the Betti cone $\mathbb{B}_{\mathbb{Q}}(R)$ if whenever there exist $c_1, \ldots,c_n\in \mathbb{Q}_{\geq 0}$ and finitely generated $R$-modules $M_1,\ldots,M_n$ such that
$\beta^R(M) = \sum_{i=1}^{n} c_i \beta^R(M_i)$, then we have  $\beta^R(M)= c \beta^R(M_i)$ for some $i \in \{1,\ldots, n\}$ and $c \in \mathbb{Q}$.

\item The condition $\BB_{\mathbb{Q}}(R)=\BB^{\text{pure}}_{\mathbb{Q}}(R)$ is equivalent to all extremal rays of $\mathbb B_{\mathbb Q}(R)$ being generated by Betti tables of pure modules.
\end{enumerate}

\begin{rmk}\label{dim0and1extremals} \cite[Theorem 4.3]{AK20}  
Let $R$ be a standard graded $\sk$-algebra with $H_R(z)= (1+nz)/(1-z)^d$, where $\dim(R) = d \leq 1$. Then $\BB_{\mathbb{Q}}(R)=\BB^{\text{pure}}_{\mathbb{Q}}(R)$. Moreover, 
\begin{enumerate}[{\rm a)}]
\item if $d=0$, then the extremal rays of $\BB_{\mathbb{Q}}(R)$ are spanned
by the Betti diagrams of the shifts of $R$, and $\sk$.
\item if $d=1$, then  the extremal rays of $\BB_{\mathbb{Q}}(R)$ are spanned
by the Betti diagrams of the shifts of $R$, and $R/\m^i$, and $R/\langle l\rangle ^i$, where $i \in \mathbb{N} $ and $l$ is a linear regular element in $R$.
\end{enumerate} 
\end{rmk}

\subsection{Fibre Products over $\sf k$}
Let $(R_1,\m_1,\sk)$, and $(R_2,\m_2,\sk)$ be standard graded $\sk$-algebras with $\m_1$ and $\m_2$ being the respective homogeneous maximal ideals, and $\epsilon_{1}:R_1 \to   \sk$, $\epsilon_{2}:R_2\to  \sk$  be the natural maps. Then the \emph{fibre product of $R_1$ and $R_2$ over $\sk$}, denoted by $R_1 \times_{\mathsf{k}} R_2$, is defined as
\[ R=R_1 \times_{\mathsf{k}} R_2 = \{ (g,h) \in R_1 \times R_2 \mid \epsilon_{1}(g) = \epsilon_{2}(h) \}. \]

For example, if $R_1= \sk[X_1,\ldots,X_n]$ and $R_2= \sk[Y_1,\ldots,Y_m]$, then we have \[ R = \sk [X_1,\ldots,X_n,Y_1,\ldots,Y_m]/\langle X_iY_j \mid 1\leq i\leq n, 1\leq j \leq m\rangle.\]

{\sc Note:} To avoid trivial cases, we assume throughout the article that $R_1 \not \simeq \sk$ and $R_2 \not \simeq \sk$. 

\begin{rmk}\label{rmk:basicProperitesOfFibreProducts}{\rm 
Let $(R_1,\m_1,\sk)$, and $(R_2,\m_2,\sk)$ be 
standard graded $\sk$-algebras and $R = R_1 \times_{\sf k} R_2$.
 \begin{enumerate}[{\rm a)}]
\item The fibre product $R$ fits into the following pullback diagram

\[ \begin{tikzcd}
R \arrow{r} \arrow{d} & R_1 \arrow{d}{\epsilon_{1}} \\
 R_2 \arrow{r}[swap]{\epsilon_{2}} & \sk 
\end{tikzcd} \]
 where $ R \to R_1$ and $ R \to R_2$ are the natural projections.

\item If $(R_1, \m_1, \sk), (R_2, \m_2, \sk)$ are standard graded $\sk$-algebras, then so is $R$, with homogeneous maximal ideal $\m_1 \oplus \m_2$.

\item There is a short exact sequence $ 0 \to R  \to R_1 \oplus R_2 \to \sk \to 0$ of $ R$-modules. Hence we have $\depth( R) = \min \{1, \depth (R_1), \depth(R_2)\}$ and $\dim( R) = \max \{ \dim (R_1), \dim(R_2)\}$.
\end{enumerate}
}\end{rmk}
 More details about the above observations can be found in \cite[Chapter 4]{HAthesis}.

\section{Syzygies Over Fibre Products}\label{Syzygies Over Fibre Products}

Throughout this section, $R_1$ and $R_2$ are standard graded $\sk$-algebras, and $R=R_1 \times_{\mathsf k} R_2$. We look at the syzygies over $R$ in this section, initially of an $R_1$-module $M$. It follows from Lemma \ref{first_syzygy_for_direct_sum} that these are the building blocks towards understanding higher syzygies of a general $R$-module $M$, which we do later in the section.

\begin{rmk}\label{notation:FP} This remark sets up the notation for the rest of the article.
\begin{enumerate}[\rm a)]
\item For $R = R_1 \times_{\mathsf{k}} R_2$, and let $\pi_j: R \rightarrow R_j$ be the natural onto maps. 

For $j=1,2$ and $t\in \mathbb{N} $, we use the same notation to denote the natural projections $\pi_j: R^{\oplus t}\longrightarrow R_j^{\oplus t}$.

\item For $j=1,2$ we identify $\m_j$ as a subset of $\m$ via the natural inclusion $\iota_j:\m_j \to \m_1\oplus \m_2$. In a similar way, we identify $\m_j^{\oplus t}$ as a subset of $\m^{\oplus t}$.

\item Let $M$ be a finitely generated $R$-module with $\mu(M) = t$, and $F=\oplus_{j=1}^t Re_j$ map minimally onto $M$. Then every element of $\Omega_1^R(M)$ can be written as $ \sum_{j=1}^t (g_j+h_j)e_j$, where $g_j\in \m_1$ and $h_j\in \m_2$.

Furthermore, if $M$ is an $R_1$-module, then $\Omega_1^{R_1}(M) \subset \m_1^{\oplus t}$ can identified with the corresponding submodule of $\m_1F$.
\end{enumerate}  
\end{rmk}

\begin{lemma}{\label{first_syzygy_for_direct_sum}} 
Let $R_1$ and $R_2$ be standard graded $\sk$-algebras, $R= R_1 \times_{\mathsf k} R_2$. If $M_1$ is minimally generated as an $R_1$-module by $\{x_1,\ldots, x_n\}$, with $\deg(x_j)=c_j$, and $M_2$ is minimally generated as an $R_2$-module by $\{y_1, \ldots, y_m\}$, with $\deg(y_j) = c_j'$, then
 \begin{enumerate}[{\rm a)}]
 \item
 $\Omega_1^R(M_1) \simeq \Omega_1^{R_1}(M_1)  \oplus \left(\bigoplus_{j=1}^{n} \m_2(-c_j)\right)$. 
 
 In particular, if $M_1$ has a minimal generator in degree $d$, 
 then $\Omega_1^R(M_1)$ has a minimal genertor in degree $d+1$.
 
 \item $$ \Omega_1^R(M_1 \oplus M_2) \simeq \Omega_1^{R_1}(M_1)  \bigoplus 
                 \left( \bigoplus_{j=1}^{n} \m_2(-c_j) \right)  \bigoplus \Omega_1^{R_2}(M_2)\bigoplus \left( \bigoplus_{j=1}^{m} \m_1{(-c'_{j}))} \right).$$
 \end{enumerate}
 
 \end{lemma}
 \begin{proof}
 Note that (b) follows from (a). We now prove (a). 
 
 Let $F= \bigoplus_{j=1}^{n} Re_j$ with $\deg(e_j)=c_j$, and $\varphi:F \to M$ be an $R$-linear map given by $\varphi(e_i) = x_i$. 
 Note that $\m=\m_1\oplus \m_2$, and $\m_2$ acts trivially on $M$. Thus, $ \Omega_1^{R_1}(M) \oplus \left(\bigoplus_{j=1}^{n} \m_2e_j\right)\subset  \Omega_1^R(M)$.  
 
 To prove the other inclusion, let $ \sum_{j} (g_j+h_j)e_j\in \Omega_1^R(M)$, where $g_j\in \m_1 $ and $h_j\in \m_2$. Since $\sum_jh_je_j\in \left(\bigoplus_{j=1}^{n} \m_2e_j\right)\in \Omega_1^R(M)$, we see that $\sum_jg_je_j\in \Omega_1^R(M)$. Therefore, $\sum_jg_jx_j=0$, and hence $\sum_jg_je_j\in \Omega_1^{R_1}(M)$.

 Finally, the second part of (a) follows by the choice of $c_j$.
 \end{proof}

The following proposition relates the syzygies of an $R_1$-module $M$ to those of $M$ as an $R$-module.

 \begin{prop}\label{FieldSyzygyLemma}
 Let $R_1$ and $R_2$ be standard graded $\sk$-algebras, $R= R_1 \times_{\mathsf k} R_2$, and $M$ be a finitely generated graded nonzero $R_1$-module. Then 
\[ \Omega^R_i(M) \simeq  \Omega^{R_1}_i(M) \oplus \Omega_i' \oplus \Omega_i'',\]
where 
 $$\Omega_i'\simeq  \bigoplus_{j=1}^i  \bigoplus_{k \in \mathbb{Z}}\Omega^{R_2}_j(\sk)^{a_{j,k}}(-d_{j,k})  \text{\rm \quad and \quad } \Omega_i''\simeq  \bigoplus_{j=1}^{i-1} \bigoplus_{k \in \mathbb{Z}}\Omega^{R_1}_j(\sk)^{b_{j,k}}(-d'_{j,k})$$
 for some $a_{j,k}\geq 0, b_{j,k}\geq 0,$   $d_{j,k}, d'_{j,k} \in \BZ$.
  \end{prop}\begin{proof} 
 
 Let \[\BF_\bullet : \cdots \xrightarrow[]{} F_n \xrightarrow[]{\phi_n} F_{n-1} \xrightarrow[]{\phi_{n-1}} \cdots \xrightarrow[]{\phi_1}F_0 \xrightarrow[]{\phi_0} M \xrightarrow[]{} 0\]
 be a minimal graded free resolution of $M$.

 By Lemma \ref{first_syzygy_for_direct_sum}(a), the result is true for $i=1$. Let $i\geq 1$, and inductively assume that  
 \[\Omega^R_i(M) \simeq \Omega^{R_1}_i(M)     \oplus \Omega_i' \oplus \Omega_i{''},\]
 where $\Omega_i'$ and $\Omega_i{''}$ are as in the statement. Hence, $\Omega_{i+1}^R(M) \simeq \Omega_1^R(\Omega_i^{R_1}(M) \oplus \Omega_i'') \oplus \Omega_1^{R}(\Omega_i')$.
 
Let $\{x_1,\ldots,x_n\}$ and  $\{y_1,\ldots,y_m\}$ be minimal generating sets of $\Omega_{i}^{R_1}(M) \oplus \Omega_i''$ and $\Omega_i'$, respectively. Suppose that $c_j=\deg(x_j)$ and $c_j'=\deg(y_j)$. 
Then by Lemma \ref{first_syzygy_for_direct_sum}(b), we get  
\[\Omega_{i+1}^R(M) \simeq \Omega_{i+1}^{R_1}(M) \oplus \Omega_1^{R_1}(\Omega_i'') \oplus \left(\bigoplus\limits_{i=1}^n\m_2(-c_j)\right) \oplus \Omega_1^{R_2}(\Omega_i')\oplus\left( \bigoplus\limits_{i=1}^m\m_1(-c_j')\right).\]
  
 Note that $\Omega_1^{R_1}(\Omega_i'')\simeq \left(
        \bigoplus_{j=1}^{i-1} \bigoplus_{k \in \mathbb{Z}}\Omega^{R_1}_{j+1}(\sk)^{b_{j,k}}(-d'_{j,k})       \right) $. We can write $\Omega_1^{R_2}(\Omega_i')$ similarly. 
        
     The result now follows since $\Omega_1^{R_1}(\sk)\simeq\m_1$ and $\Omega_1^{R_2}(\sk)\simeq \m_2$.
 \end{proof}

\begin{rmk}\label{rmk:Omegas}
    Let the notation be as in Proposition \ref{FieldSyzygyLemma} and its proof.
    
    \begin{enumerate}[{\rm a)}]
        \item If $$\Omega_i'\simeq  \bigoplus_{j=1}^i  \bigoplus_{k \in \mathbb{Z}}\Omega^{R_2}_j(\sk)^{a_{j,k}}(-d_{j,k})  \text{\rm \quad \text{ and} \quad } \Omega_i''\simeq  \bigoplus_{j=1}^{i-1} \bigoplus_{k \in \mathbb{Z}}\Omega^{R_1}_j(\sk)^{b_{j,k}}(-d'_{j,k})$$

    then 
     $$\Omega_{i+1}'\oplus \Omega_{i+1}''  \simeq    \left( 
                 \bigoplus_{j=1}^{i}  \bigoplus_{k \in \mathbb{Z}}\Omega^{R_2}_{j+1}(\sk)^{a_{j,k}}(-d_{j,k})    \right) \bigoplus   \left(
                 \bigoplus_{j=1}^{i-1} \bigoplus_{k \in \mathbb{Z}}\Omega^{R_1}_{j+1}(\sk)^{b_{j,k}}(-d'_{j,k})       \right)   $$
                 $$\qquad\qquad\qquad\qquad\qquad\qquad\qquad\qquad\qquad \bigoplus \left( \bigoplus_{j=1}^{n} \m_2{(-c_{j}))} \right) \bigoplus 
                 \left( \bigoplus_{j=1}^{m} \m_1(-c'_j) \right).$$
\item Moreover, if $\Omega_i^R(M)$ has a minimal generator in degree $d$, then $\m_1(-d)$ or $\m_2(-d)$ is a direct summand of $\Omega_{i+1}'\oplus \Omega_{i+1}''$.  In particular, $\Omega_{i+1}'\oplus \Omega_{i+1}''$, and hence $\Omega_{i+1}^R(M)$ have minimal generators in degree $d+1$. Thus, $\beta^{R_1}_{i,d}(M)\neq 0$, then $\beta^R_{i+j,d+j}(M) \neq 0$ for all $j \geq 1$.
    \end{enumerate}         
\end{rmk}

Another immediate consequence of Proposition \ref{FieldSyzygyLemma} is the following.
\begin{cor}\label{Cor-inequality-of-reg}
Let $R_1$ and $R_2$ be standard graded $\sk$-algebras, $R= R_1 \times_{\mathsf k} R_2$, and $M$ be a finitely generated graded  $R_1$-module. Then $\Omega^{R_1}_i(M)\mid \Omega^R_i(M)$ for each $i$, and as a consequence, we have 
\begin{enumerate}[{\rm a)}]
\item $\reg_R(M) \geq \reg_{R_1}(M)$, and 
\item if $M$ has a linear resolution over $R$, then it has a linear resolution over $R_1$.
\end{enumerate}
\end{cor}

We now study the syzygies of a finitely generated graded module $M$ over $R= R_1 \times_{\mathsf k} R_2$. In particular, in Corollary \ref{cor:splitSyzygy}, we see that $\Omega^R_2(M)$ decomposes as a direct sum of modules over $R_1$ and $R_2$, respectively. We first prove the following technical results.

\begin{lemma}{\label{kernel_of_pi_1}} Let $R_1$ and $R_2$ be standard graded $\sk$-algebras, $R= R_1 \times _{\mathsf k} R_2$, $F$ be a free $R$-module of finite rank, and $\pi_1$ and $\pi_2$ be as in Remark \ref{notation:FP}. Let $N \subset \m F$ be a submodule, with a minimal generating set $\{x_1,\ldots,x_r, y_1,\ldots,y_s, z_1, \ldots, z_t\}$, where $x_i \in \m_1F$, $y_j \in \m_2F$, and $z_k \in \m F \setminus (\m_1F \cup \m_2F)$ are such that $t$ is minimum. Then, with $\underline x = x_1, \ldots, x_r$, $\underline y = y_1, \ldots, y_s$, and $\underline z = z_1, \ldots, z_t$, we have:
\begin{enumerate}[{\rm a)}]
\item $\langle \underline y\rangle \subset \ker(\pi_1) \cap N \subset \langle \underline y\rangle + \m N$, and $\langle \underline x\rangle \subset \ker(\pi_2) \cap N \subset \langle \underline x \rangle + \m N$.
\item $\pi_1(N)$ and $\pi_2(N)$ are minimally generated by $\{ \underline x, \pi_1(\underline z) \}$, and $\{ \underline y, \pi_2(\underline z)\}$ respectively.
\end{enumerate}
\end{lemma}
\begin{proof}
Write $z_k = x_{r+k} + y_{s+k}$, where $x_{r+k} \in \m_1F \setminus \{0\}$, and $y_{s+k} \in \m_2F \setminus \{0\}$ for $k \in \{1,\ldots, t\}$. Note that $\pi_1(z_k) = x_{r+k}$, and $\pi_2(z_k) = y_{s+k}$ for all $k$. Moreover, if $(a,b) \in R$, then $(a,b)x_i=ax_i$, $(a,b)y_j=by_j$, and $(a,b)z_k=ax_{r+k}+by_{s+k}$.

a) Clearly, by the definition of $\pi_1$, we have $\langle \underline y \rangle \subset \ker(\pi_1) \cap N$. For $(a_i, a_i'), (b_j, b_j'), (c_k, c_k') \in R$, suppose $\sum \limits_{i = 1}^r(a_i, a_i')x_i + \sum \limits_{j = 1}^s(b_j, b_j')y_j + \sum \limits_{k = 1}^t(c_k, c_k')z_k \in \ker(\pi_1) \cap N$. Applying $\pi_1$, we get $\sum \limits_{i=1}^{r} a_ix_i + \sum\limits_{k=1}^t c_{k}x_{r+k} =0$. 

We first claim that for all $k \in \{1, \ldots, t\}$, we have $c_{k} \in \m_1$. If not, without loss of generality, we may assume that $c_{t} \in R_1 \setminus \m_1$, i.e., $c_{t}$ is a unit in $R_1$. Note that $(c_t, c_t') \in R$ forces $c_t' \in R_2 \setminus \m_2$. 

Write $x_{r+t} = -c_t^{-1}\left(\sum\limits_{i = 1}^r a_ix_i + \sum\limits_{k = 1}^{t - 1} c_kx_{r+ k}\right)$. Then $\{\underline x, \underline y, z_1,\ldots, z_{t-1}, y\}$ is a minimal generating set of $M$, where $y = z_t + (c_t,c_t')^{-1} \left(\sum \limits_{i = 1}^r (a_i,a_i') x_i + \sum\limits_{k = 1}^{t - 1} (c_k, c_k') z_{k} \right) \in \m_2 F$. This contradicts the minimality of $t$.

Thus, $c_k \in \m_1$ for all $k$, and hence $(c_k, 0) \in R$ for all $k$. Therefore, $\sum \limits_{i=1}^{r} a_ix_i + \sum\limits_{k=1}^t c_{k}x_{r+k} =0$ in $\pi_1(F)$ implies that $\sum\limits_{i=1}^r (a_i,a_i')x_i + \sum\limits_{k=1}^t (c_k,0)z_k \in \ker(\pi_1)$. Also, since $a_i'x_i=0$ for all $i$, we see that  $\sum\limits_{i=1}^r (a_i,a_i')x_i + \sum\limits_{k=1}^t (c_k,0)z_k \in \ker(\pi_2)$.
Thus, we get $\sum\limits_{i=1}^r (a_i,a_i')x_i + \sum\limits_{k=1}^t (c_k,0)z_k = 0$ in $F$.

Since $\underline x, \underline z$ is a part of a minimal generating set of $N$, this forces $(a_i, a_i') \in \m$. 
Finally, $c_k \in \m_1$ forces $c_k' \in \m_2$, and hence $(c_k, c_k') \in \m$ for all $k$. 

Thus, $\sum \limits_{i = 1}^r(a_i, a_i')x_i +  \sum \limits_{k = 1}^t(c_k, c_k')z_k \in \m N$, completing the proof of (a), as the second part follows similarly.

b) It is clear that $\{ \underline x, \pi_1(\underline z) \}$, and $\{ \underline y, \pi_2(\underline z)\}$ are generating sets of $\pi_1(N)$ and $\pi_2(N)$, respectively. In order to prove that they are minimal, it is enough to show it in the former case, as the latter follows similarly. 

Suppose $\sum \limits_{i=1}^{r+t} a_ix_i \in\m_1\pi_1(N)$ for some $a_i \in R_1$. We want to prove that $a_i \in \m_1$ for all $i$. Write $\sum \limits_{i=1}^{r+t} a_ix_i = \sum\limits_{i = 1}^{r+t} b_i x_i$ for some $b_i \in \m_1$, i.e., $\sum \limits_{i=1}^{r+t} (a_i - b_i) x_i = 0$. Choose $a_i', b_i' \in R_2$ such that $(a_i,a_i')$, $(b_j, b_j') \in R$. Observe that $b_j \in \m_1$ implies that $b_j' \in \m_2$, and hence $(b_j,b_j') \in \m$.

Then $\sum\limits_{i=1}^r (a_i-b_i, a_i'-b_i')x_i + \sum\limits_{k=r+1}^{r+t} (a_k - b_k, a_k' - b_k')z_k \in \ker(\pi_1)$. The proof is complete by the following remark.
\end{proof}

\begin{rmk}{\rm
The proof of (a) in the above lemma yields the following stronger statement:\\ 
If $\sum \limits_{i = 1}^r(a_i, a_i')x_i + \sum \limits_{k = 1}^t(c_k, c_k')z_k \in \ker(\pi_1)$, then $(a_i, a_i'), (c_k,c_k') \in \m$, and hence $a_i, c_k \in \m_1$.
}\end{rmk}

The following proposition describes the second syzygy of an $R$-module, in terms of related modules over $R_1$ and $R_2$.

\begin{prop}{\label{first_syzygy}} Let $R_1$ and $R_2$ be standard graded $\sk$-algebras, $R= R_1 \times_{\mathsf k} R_2$, and $N$ be a finitely generated graded $R$-module. Suppose $F$ is a free $R$-module of finite rank, and $N \subset \m F$.  Then with $\pi_1, \pi_2$ as in Remark \ref{notation:FP}, we have
 $$ \Omega_1^R(N) \simeq \left(\bigoplus_{i=1}^r\m_2({d_i})\right)\oplus\Omega_1^{R_1}(\pi_1(N)) \oplus  \Omega_1^{R_2}(\pi_2(N)) \oplus \left(\bigoplus_{j=1}^s\m_1({d_j'})\right) ,$$
where $r = \mu(N)- \mu(\pi_2(N))$, $s=\mu(N)- \mu(\pi_1(N))$, $d_i=\deg(x_i)$ for $1\leq i \leq r$, and $d_j'=\deg(y_j)$ for $1 \leq j \leq s$, with $x_i, y_j$ as in Lemma \ref{kernel_of_pi_1}.
 \end{prop}
 \begin{proof}
 Let $\{\underline x, \underline y, \underline z\}$ be a minimal generating set of $N$, with notation as in the previous lemma. Suppose $G_1$, $G_2$, $G_3$ are free $R$-modules with bases $\{e_1,\ldots, e_{r}\}$, $\{ f_1,\ldots,f_{t}\}$, $\{g_1,\ldots,g_{s}\}$ respectively, and $G= G_1 \oplus G_2 \oplus G_3$. Consider an $R$-module homomorphism $\phi: G\to N$ given by $$\phi(e_i)=x_i, \phi(f_k)= z_k = x_{r+k} + y_{s+k}, \text{ and } \phi(g_j)=y_j.$$ Then $\Omega_1^R(N) \simeq \ker(\phi) \subset \m G = \m_1 G_1 \oplus \m_2 G_1 \oplus \m_1 G_2 \oplus \m_2 G_2 \oplus \m_1 G_3 \oplus \m_2 G_3$. 

{\it Claim:} $\Omega_1^{R_1}(\pi_1(N))$ and $\Omega_1^{R_2}(\pi_2(N))$ can both be identified with submodules of $\ker(\phi)$. 

We prove the claim for $\Omega_1^{R_1}(\pi_1(N))$, and the other case is similar.
In order to prove this, recall that $\pi_1(N)$ is minimally generated by $\{x_1,\ldots, x_r, x_{r+1}, \ldots, x_{r+t}\}$, by Lemma \ref{kernel_of_pi_1}(b). Hence, $\pi_1(G_1 \oplus G_2)$ is a free $R_1$-module which maps minimally onto $\pi_1(N)$ by a map, say $\phi_1$, induced from $\phi$. Then $\Omega_1^{R_1}(\pi_1(N)) \simeq \ker(\phi_1) \subset \m_1\pi_1(G_1 \oplus G_2)$. Identifying the last module with $ \m_1G_1 \oplus \m_1G_2$ as in Remark \ref{notation:FP}, we now show that $\ker(\phi_1)\subset \ker(\phi)$.

Now, assume $(a_1,\ldots,a_r, b_1,\ldots,b_t)\in \ker(\phi_1)\subset  \m_1G_1 \oplus \m_1G_2$. Then $\sum_{i=1}^r a_ix_i + \sum_{k=1}^t b_k x_{r+k} =0$, and hence 
 $$\phi\left( \sum\limits_{i=1}^{r}(a_i,0)e_i + \sum\limits_{k=1}^{t} (b_k,0) f_k \right) = \sum\limits_{i=1}^{r}(a_i,0)x_i + \sum\limits_{k=1}^{t} (b_k,0) (x_{r+k}+y_{s+k})=0.$$ 
 Thus, $(a_1,\ldots, a_r, b_1, \ldots, b_t) \in \ker(\phi)$, and hence $\ker(\phi_1)\subset \ker(\phi)$, proving the claim. 
 
Thus, to prove the proposition, it is enough to show the following: 
\[\m_2G_1\oplus\ker(\phi_1) \oplus\ker(\phi_2) \oplus \m_1G_3= \ker(\phi).\] 

We have proved that $\ker(\phi_1)\oplus \ker(\phi_2)\subset \ker(\phi)$. Now we get $\phi(\m_2G_1) =\m_2 \langle x_1, \ldots, x_r\rangle = 0$. Similarly, since $\m_1y_j=0$, we get $\phi(\m_1G_3) = 0$. This proves 
\[\m_2G_1\oplus\Omega_1^{R_1}(\pi_1(N)) \oplus  \Omega_1^{R_2}(\pi_2(N)) \oplus \m_1G_3\subset \Omega_1^R(N).\] 

To prove the other inclusion, let 
 \[\sigma =\sum \limits_{i=1}^{r}(a_i+a_i')e_i + \sum \limits_{k=1}^{t} (b_{k}+b_k') f_k + \sum\limits_{j=1}^{s} (c_j+c_j') g_j \in \ker(\phi),\]
 where $a_i, b_k, c_j \in \m_1$, and $a_i', b_k', c_j' \in \m_2$. 
 
Thus, $\phi(\sigma)=0$ gives 
\[ \sum \limits_{i=1}^{r} (a_i + a_i')x_i + \sum \limits_{k=1}^{t} (b_{k}+b_k') z_k+ \sum\limits_{j=1}^{s} (c_j+c_j') y_j=0.\] 

Now, using $x_i \in \m_1 F$ for $i \in \{1, \ldots, r+t\}$, $y_j \in \m_2 F$ for $j \in \{1, \ldots, s+t\}$, $z_k = x_{r+k} + y_{s+k}$ for all $k$, and $\m_1 \m_2 = 0$, the above equation can be rewritten as 
\[ \left(\sum \limits_{i=1}^{r}a_i'x_i\right)+\left(\sum \limits_{i=1}^{r}a_ix_i + \sum \limits_{k=1}^{t} b_{k} x_{r+k}\right) + \left( \sum \limits_{k=1}^{t} b_{k}' y_{s+k}+ \sum\limits_{j=1}^{s} c_j' y_j\right) + \left(\sum \limits_{j=1}^{s}c_jy_j\right)=0, \]

where the first and the last terms are zero. Moreover, $\m F = \m_1 F \oplus \m_2 F$ forces the second and the third terms to be individually zero. Thus, we have

 \[\sum \limits_{i=1}^{r}a_i'x_i =0, \ \   \sum \limits_{i=1}^{r}a_ix_i + \sum \limits_{k=1}^{t} b_{k} x_{r+k}=0,\  \ \sum \limits_{k=1}^{t} b_{k}' y_{s+k}+ \sum\limits_{j=1}^{s} c_j' y_j=0,\ \ \sum \limits_{j=1}^{s}c_jy_j=0. \]
Therefore, we get that
\[\sigma =\left(\sum \limits_{i=1}^{r}a_i'e_i \right)+ \left(\sum\limits_{i=1}^{r}a_ie_i+\sum \limits_{k=1}^{t} b_{k} f_k\right) + \left(\sum\limits_{k=1}^{t}b_k'f_k +\sum\limits_{j=1}^{s} c_j' g_j\right) + \left(\sum \limits_{j=1}^{s}c_jg_j\right)\]
is in $ \m_2G_1\oplus\Omega_1^{R_1}(\pi_1(N)) \oplus  \Omega_1^{R_2}(\pi_2(N)) \oplus \m_1G_3$. 
This completes the proof.
 \end{proof}

Taking $N = \Omega_{i-1}^R(M)$ in Proposition \ref{first_syzygy}, as an immediate application, we get the following.
\begin{cor}\label{cor:splitSyzygy}
Let $R = R_1 \times_{\mathsf k} R_2$ be a fibre product of standard graded $\sk$-algebras, and $M$ be a finitely generated graded $R$-module. Then, for $i \geq 2$, $\Omega_i^R(M)$ splits as a direct sum of modules over $R_1$ and $R_2$, respectively.
\end{cor}
 
\begin{rmk}\label{directSummandRemark}{\rm 
Let $R= R_1 \times_{\mathsf k} R_2$ be a fibre product of standard graded $\sk$-algebras. If $M$ is an $R$-module of infinite projective dimension, then, for $i \geq 3$, there exist $k\in \{1,2\}$ and $j \in \mathbb Z$ such that $\m_k(-j) \mid \Omega_i^R(M)$. This follows from Corollary \ref{cor:splitSyzygy}, and Lemma \ref{first_syzygy_for_direct_sum} (with $M$ replaced by $\Omega_{i-1}(M)$). In fact, we have the following stronger statement: 
       
If $\m_1\mid \Omega_3^R(M)$, then $\m_1\mid \Omega_{2i-1}^R(M)$ and $\m_2\mid \Omega_{2i}^R(M)$ for all $i\geq 2$, and vice versa.
}\end{rmk}
  
\begin{example}\hfill{}
 
\begin{enumerate}[{\rm a)}]
\item We first see an example in which neither $\m_1\mid \Omega_2^R(M)$ nor $\m_2\mid \Omega_2^R(M)$.

Let $R_1=\sk[X_1,X_2], R_2= \sk[Y_1,Y_2]$ and 
\[R=R_1 \times_{\mathsf{k}} R_2 \simeq \sk[X_1,X_2,Y_1,Y_2]/\langle X_1Y_1,X_2Y_1,X_1Y_2,X_2Y_2\rangle\] 
and $M=R/\langle X_1^2+Y_1^2, X_2^2+Y_2^2\rangle $. Then,  we have $\Omega_1^R(M)\simeq \langle X_1^2+Y_1^2, X_2^2+Y_2^2\rangle $. Therefore, in the notation of Proposition \ref{first_syzygy}, we have $r=s=0$ and 
\[\Omega^R_2(M)\simeq \Omega^{R_1}_2(R_1/\langle X_1^2,X_2^2\rangle)\oplus \Omega^{R_2}_2(R_2/\langle Y_1^2,Y_2^2\rangle).\]

Clearly $\m_k\nmid \Omega^R_2(M)$ for $k = 1,2$.

\item In this example, we see that for $i\geq 1$, there can exist exactly one $k\in\{ 1,2\}$ such that $\m_k\mid \Omega_i^R(M)$. 
    
Let $R= \sk[X,Y]/\langle XY \rangle$. Then $R \simeq \sk [X]\times_{\mathsf k} \sk [Y]$. Take $M=R/\langle X\rangle$. Then $\Omega_i^R(M) \simeq \langle X \rangle = \m_1$ when $i$ is odd, and  $\Omega_i^R(M) \simeq \langle Y \rangle = \m_2$ when $i$ is even.
\end{enumerate}
\end{example}

\section{Fibre Products and the Koszul Property}
This section focuses on the Koszul property of the fibre product $R=R_1 \times_{\mathsf k} R_2$. In Proposition \ref{Prop-prep-for-koszul}, we show that an $R_1$-module having pure resolutions over $R$ must have a linear resolution over $R$, and that the existence of such modules forces $R_1$ and $R_2$ to be Koszul. As applications of this, in Corollary \ref{KoszulCorollary} and Corollary \ref{UniversallyKoszulCorollary}, we get that $R$ is (universally) Koszul if and only if the same holds for both $R_1$ and $R_2$. We end this section by showing that the fibre product $R$ being Koszul is equivalent to the existence of module with infinite projective dimension and finite regularity. 

We begin with the following technical lemma.
\begin{lemma}\label{lemma:degreeofminimalgenerators}
Let $R_1$ and $R_2$ be Koszul algebras, $R= R_1 \times_{\mathsf k}R_2$, and $M$ be a finitely generated graded $R_1$-module. 
Then with the notation as in the Proposition \ref{FieldSyzygyLemma}, $\Omega_{i+1}' \oplus \Omega_{i+1}''$ has a minimal generator in degree $j+1$ if and only if $\Omega_i^R(M)$ has a minimal generator in degree $j$. 
\end{lemma}
\begin{proof}
  $(\Leftarrow)$ This follows from Remark \ref{rmk:Omegas}(b).
  
  $(\Rightarrow)$ 
  By Remark \ref{rmk:Omegas}(a), we have $\Omega_{i+1}'\oplus \Omega_{i+1}''\simeq$ 
  $$   \left( 
                 \bigoplus_{j=1}^{i}  \bigoplus_{k \in \mathbb{Z}}\Omega^{R_2}_{j+1}(\sk)^{a_{j,k}}(-d_{j,k})    \right) \oplus   \left(
                 \bigoplus_{j=1}^{i-1} \bigoplus_{k \in \mathbb{Z}}\Omega^{R_1}_{j+1}(\sk)^{b_{j,k}}(-d'_{j,k})       \right)   \oplus \left( \bigoplus_{j=1}^{n} \m_2{(-c_{j}))} \right) \oplus 
                 \left( \bigoplus_{k=1}^{m} \m_1(-c'_k) \right),$$
                 where $C= \{ c_j, c_k' \mid 1\leq j \leq n, 1 \leq k \leq m\}$ is the set of degrees of the minimal generators of $\Omega_i^R(M)$.

Note that the minimal generators of $\m_1$ and $\m_2$ are in degree $1$, and since $R_1, R_2$ are Koszul, the minimal generators of $\Omega_{j+1}^{R_1} (\sk), \Omega_{j+1}^{R_2} (\sk)$ are in degree $j+1$. Hence, if $d$ is the degree of a minimal generator of $\Omega_{i+1}' \oplus \Omega_{i+1}''$, then $d$ is of the form $d_{j,k}+j+1, d'_{j,k}+j+1, c_j+1$, or $c_k'+1$. 

We want to show that $d-1 \in C$. This is clear if $d=c_j+1$ or $c_k'+1$. Suppose $d=d_{j,k}+j+1$. Then $\Omega_{j+1}^{R_2}(\sk)(-d_{j,k})\mid (\Omega_{i+1}' \oplus \Omega_{i+1}'')$, and hence $\Omega_{j}^{R_2}(\sk)(-d_{j,k})\mid \Omega_i^R(M)$, forcing $d_{j,k}+j \in C$. A similar argument works in the last remaining case, proving the result.
\end{proof}

The following proposition, and its corollary, compare the Betti table of an $R_1$-module $M$, with its Betti table over $R$.

\begin{prop}\label{prop:firstNonzeroEntries}
Let $R_1$ and $R_2$ be Koszul algebras, $R= R_1 \times_{\mathsf k}R_2$, and $M$ be a finitely generated graded $R_1$-module. 
\begin{enumerate}[{\rm a)}]
    \item If $\beta^R_{i,j}(M)=0$, then $\beta_{i+1,j+1}^R(M)= \beta_{i+1,j+1}^{R_1}(M)$. \\
In particular, the first nonzero entries in each row of $\beta^R(M)$ and $\beta^{R_1}(M)$ are equal.

\item Furthermore, if $M$ is a pure $R_1$-module of type $\delta = (\delta_0, \delta_1,\delta_2,\ldots)$, then we have $\beta_{i,\delta_i}^R(M) \neq 0$ and $\beta_{i,\delta_i+j}^R(M) = 0$ for all $i< \pdim_{R_1}(M) +1$ and $j \geq 1$, i.e., every minimal generator of $\Omega_i^R(M)$ has degree at most $\delta_i$.
\end{enumerate}
\end{prop}

\begin{proof}
 Since $\beta_{0,j}^R(M)= \beta^{R_1}_{0,j}(M)$ for all $j$, both (a) and (b) hold for $i =0$. Hence assume $i > 0$. 
 
 a) By Proposition \ref{FieldSyzygyLemma}, we have $\Omega^R_{i+1}(M) =  \Omega^{R_1}_{i+1}(M) \oplus \Omega_{i+1}' \oplus \Omega_{i+1}''$. By Lemma \ref{lemma:degreeofminimalgenerators}, since $\beta_{i,j}^{R}(M) =0$, we see that $\Omega'_{i+1} \oplus \Omega''_{i+1}$ has no minimal generator of degree $j+1$. This proves (a). 

b) By induction, let us assume that every minimal generator of $\Omega_{i-1}^R(M)$ has degree at most $\delta_{i-1}$. Hence, by Lemma \ref{lemma:degreeofminimalgenerators}, the minimal generators of $\Omega_{i}' \oplus \Omega_{i}''$ have degree at most $1 + \delta_{i-1} \leq \delta_i$. Part (b) follows since $\Omega_i^{R_1}(M)$ is generated in degree $\delta_i$, and $\Omega^R_{i+1}(M) =  \Omega^{R_1}_{i+1}(M) \oplus \Omega_{i+1}' \oplus \Omega_{i+1}''$.
\end{proof}

\begin{cor}\label{cor:rows_of_Betti_table}
 Let $M$ be an $R_1$-module such that the $j^{th}$ row of $\beta^R(M)$ is nonzero. Then 
 \begin{enumerate}[{\rm a)}]
    \item The $j^{th}$ row of $\beta^{R_1}(M)$ is nonzero.
    \item If $i_0 = \min\{i \mid \beta_{i, j+i}^R(M) \neq 0\}$, then $\beta_{i,j+i}^R(M) \neq 0$ for all $i \geq i_0$.
 \end{enumerate}
\end{cor}
\begin{proof}
    By the previous proposition, we have $\beta_{i_0,j+i_0}^R(M)= \beta_{i_0,j+i_0}^{R_1}(M)$, which implies (a). This, together with Remark \ref{rmk:Omegas}(b), gives (b). 
\end{proof}

\begin{prop}\label{Prop-equality-of-reg}
If $R_1$ and $R_2$ are Koszul, then we have the following:

\begin{enumerate}[{\rm a)}]

\item If $M$ is an $R_1$-module, then $\reg_{R_1}(M) = \reg_R(M)$. 

\item If $M_j$ is an $R_j$-module for $j = 1,2$, then $\reg_R(M_1\oplus M_2) = \max\{\reg_{R_1}(M_1), \reg_{R_2}(M_2)\}.$

\end{enumerate}
\end{prop}
 \begin{proof}

a) By Corollary \ref{Cor-inequality-of-reg}, we have $\reg_R(M) \geq \reg_{R_1}(M)$. 

In order to prove the other inequality, we may assume that 
$\reg_{R_1}(M)= s <\infty$. This implies that the $j^{th}$ row of $\beta_{R_1}(M)$ is zero for all $j >s$, and hence the same holds for $\beta^R(M)$ by Corollary \ref{cor:rows_of_Betti_table}(a). Hence, $\reg_R(M) \leq s$. 

b) This follows from (a), since $\reg_R(M_1\oplus M_2)= \max\{ \reg_R(M_1) , \reg_R(M_2)\}$.
\end{proof} 

 \begin{prop}\label{Prop-prep-for-koszul}
 Let $R_1$ and $R_2$ be standard graded $\sk$-algebras, $R=R_1 \times_{\mathsf k} R_2$, and $M$ a nonzero $R_1$-module. Then the following are equivalent:
 \begin{enumerate}[{\rm i)}]
     \item $M$ has a pure resolution over $R$.
     \item $M$ has a linear resolution over $R$.
     \item $M$ has a linear resolution over $R_1$, and $R_1$, $R_2$ are Koszul.
 \end{enumerate}
 \end{prop}
 \begin{proof}
(i) $\Rightarrow$ (ii): Let $M$ be generated in degree $d$. Then by Lemma \ref{first_syzygy_for_direct_sum}(a), we have $\beta_{i,d+i}^R(M)\neq 0$ for all $i \geq 1$. Since $M$ has a pure resolution over $R$, this forces $\beta_{i,j}^R(M) = 0$ for $j \neq d + i$, i.e., $M$ has a linear resolution over $R$.

(ii) $\Rightarrow$ (iii): By Corollary \ref{Cor-inequality-of-reg}(b), (ii) implies that $M$ has a linear resolution over $R_1$. 

Moreover, if $M$ is generated in degree $d$, then by Lemma \ref{first_syzygy_for_direct_sum}(a), we have $\m_2(-d) \mid \Omega_1^R(M)$, and hence $\m_1(-d-1) \mid \Omega_2^R(M)$. Since $M$ has a linear resolution over $R$, we see that both $\m_1$ and $\m_2$ have linear resolutions over $R$. In particular, by Corollary \ref{Cor-inequality-of-reg}(b), for each $j \in \{1,2\}$, we see that $\m_j$ has a linear resolution over $R_j$, proving that $R_j$ is Koszul.

(iii) $\Rightarrow$ (ii) (and hence (i)): Follows from Proposition \ref{Prop-equality-of-reg}(a).
 \end{proof}
 
 A finer version of Proposition \ref{Prop-prep-for-koszul} is true, which follows from its proof. We capture this in the following remark.  
 \begin{rmk}
Let $M$ be a nonzero $R_1$-module. Then the following are equivalent:
\begin{enumerate}[{\rm i)}]
    \item $M$ is linear up to the $(i+1)^{st}$ stage over $R$. 
    \item $M$ is linear up to the $(i+1)^{st}$ stage over $R_1$, $\mathsf k$ is linear up to the $i^{th}$ stage over $R_2$, and up to the $(i-1)^{st}$ stage over $R_1$.
\end{enumerate}  
\end{rmk}

The following corollary is immediate from the proof of Proposition \ref{Prop-prep-for-koszul} applied to $M=\sk$. The equivalence of (ii) and (iii) was proved by Backelin and Fr\"oberg in \cite{BF85}.
 \begin{cor}\label{KoszulCorollary}
  The following are equivalent:
 \begin{enumerate}[{\rm i)}]
     \item $\sk$ has a pure resolution over $R$.
     \item $R$ is Koszul, i.e., $\sk$ has a linear resolution over $R$.
     \item $R_1$ and $R_2$ are Koszul.
 \end{enumerate}
 \end{cor}

As a concrete application of Proposition \ref{first_syzygy} and Proposition \ref{Prop-prep-for-koszul}, we obtain the following result, which was also proved by Conca in \cite{AC2000}.
\begin{cor}\label{UniversallyKoszulCorollary}
    Let $R_1$ and $R_2$ be standard graded $\sk$-algebras and $R=R_1 \times_{\mathsf k} R_2$. Then $R$ is universally Koszul if and only if $R_1$ and $R_2$ are universally Koszul.
\end{cor}
\begin{proof}
    Let $R$ be universally Koszul. Hence, if $I\subset R_j$ is an ideal generated by linear forms, we know that $I$ has a linear resolution over $R$. By Proposition \ref{Prop-prep-for-koszul}, we get that $I$ has a linear resolution over $R_j$. This shows that $R_j$ is universally Koszul.

    Conversely, suppose that $R_1$ and $R_2$ are universally Koszul. 
    Let $I= \langle x_1+y_1, x_2+y_2, \ldots, x_n+y_n\rangle$ be an ideal in $R$ generated by linear forms with $x_i\in \m_1$, $y_j\in \m_2$. By Proposition \ref{first_syzygy}, we have $\Omega_1^R(I) = \Omega_1^{R_1}(\langle x_1, \ldots, x_n\rangle) \oplus \Omega_1^{R_2}(\langle y_1, \ldots, y_n\rangle) \oplus \m_2^{\oplus r}(-1) \oplus \m_1^{\oplus s}(-1) $ for some $r, s \geq 0$. Since $R_1$ and $R_2$ are universally Koszul, the modules $\Omega_1^{R_1}(\langle x_1, \ldots, x_n\rangle)$ and  $\Omega_1^{R_2}(\langle y_1, \ldots, y_n\rangle)$ are generated in degree 2, and have a linear resolution. Moreover, since $\reg_{R_j}(\m_j)=1$, by Corollary \ref{Cor-inequality-of-reg}, we have $\reg_R(\m_j)=1$. This shows that $\Omega_1^R(I)$ is generated in degree 2, and has a linear resolution, which implies that $I$ has a linear resolution over $R$.  
\end{proof}

Avramov and Peeva (\cite{AP01}) proved that a $\sk$-algebra $R$ is Koszul if and only if every finitely generated graded $R$-module has finite regularity. For fibre products, we have the following stronger result.
\begin{prop}\label{prop:koszulChar}
Let $R_1$ and $R_2$ be standard graded $\sk$-algebras, and $R=R_1 \times_{\mathsf k} R_2$. Then the following are equivalent:
\begin{enumerate}[{\rm i)}]
    \item There exists an $R$-module $M$ with $\pdim_R(M)=\infty$ and $\reg_R(M)<\infty$.
    \item $\reg_R(\sk)< \infty$.
    \item $R$ is Koszul.
\end{enumerate}
\end{prop}
\begin{proof}
(i) $\Rightarrow$ (ii): By Remark \ref{directSummandRemark}, either $\m_1(-\delta)\mid \Omega_3^R(M)$ or $\m_1(-\delta)\mid \Omega_4^R(M)$ for some $\delta\in \mathbb Z$. Since $\reg_R(M)<\infty$, we get $\reg_R(\m_1)<\infty$. Similarly, $\reg_R(\m_2)<\infty$. Thus, $$\reg_R(\sk)=\reg_R(\m_1 \oplus \m_2) -1 = \max\{\reg_R(\m_1), \reg_R(\m_2)\} -1 <\infty.$$ 

(ii) $\Rightarrow$ (iii): Suppose $\reg_R(\sk)=s<\infty$, and let $\beta^R_{i,s+i}(\sk) \neq 0$. 
By Remark \ref{rmk:Omegas}(b), since $\Omega_i^R(\sk)$ has a minimal generator in degree $s+i$, we have $\m_1(-(s+i)) \mid \Omega_{i+1}^R(\sk)$ or $\m_2(-(s+i)) \mid \Omega_{i+1}^R(\sk)$. Without loss of generality, let $\m_1(-(s+i)) \mid \Omega_{i+1}^R(\sk)$.\\ 
Then $\reg_R(\m_1(-(s+i))) \leq \reg_R\left(\Omega_{i+1}^R(\sk)\right)\leq s+i+1$, where the last inequality is true since $\reg_R(\mathsf k) = s$. This forces $\reg_R(\m_1)=1$.

Since $\m_1(-(s+i)) \mid \Omega_{i+1}^R(\sk)$, we have $\m_2(-(s+i+1)) \mid \Omega_{i+2}^R(\sk)$. The same argument as above shows that $\reg_R(\m_2)=1$.\\ Thus we see that $\Omega_1^R(\sk)= \m_1 \oplus \m_2$ has a linear resolution over $R$. Hence $R$ is Koszul.

(iii) $\Rightarrow$ (i) is clear by taking $M = \mathsf k$. 
\end{proof}

\begin{example}{\rm Note that the implication (i) $\Rightarrow$ (ii) in  Proposition \ref{prop:koszulChar} is not true without the assumption that $R$ is a fibre product. For instance, if $R= \sk[X,Y]/\langle X^2, Y^3\rangle$, then $M=R/\langle X \rangle$ has a linear resolution over $R$ with $\pdim_R(M)=\infty$, but $\reg_R(\sk)=\infty$.
}\end{example}

\section{Betti Cones over Fibre Products}\label{Betti-Cone-of-Fibre-Products}

In this section, we study the equality $\BB_{\mathbb{Q}}^{\text{pure}}(R) = \BB_{\mathbb{Q}}(R)$. In Theorem \ref{second_syzygy_linear}, we show that for a pure $R$-module $M$, its second syzygy has a linear resolution. As a consequence, this says that the regularity of a pure $R$-module $M$ is determined by the first three columns of $\beta^R(M)$. In Corollary \ref{cor:polyFibreProduct} and Proposition \ref{prop:gorensteinArtinian}, we give two classes of rings for which $\BB_{\mathbb{Q}}^{\text{pure}}(R) \neq \BB_{\mathbb{Q}}(R)$.
Finally, we prove our main result, Theorem \ref{MaintheoremFibreProducts}, which says that if $R$ is a fibre product with $\depth(R)= 1$ and $\BB_{\mathbb{Q}}^{\text{pure}}(R) = \BB_{\mathbb{Q}}(R)$, then $R$ is Cohen-Macaulay with $H_R(z)=(1+nz)/(1-z)$.

\begin{prop}\label{ExistenceProposition}
 Let $R=R_1\times_{\mathsf{k}} R_2$ be a Koszul $\sk$-algebra with $\depth(R_2)\geq 1$. Suppose there exists a pure $R_1$-module of type $(\delta_0,\delta_1,\delta_2,\ldots)$, where either $\delta_j=\delta_{j-1}+1$ or $\delta_j=\infty$ for all $j\geq 3$. Then there exists a pure $R$-module of type 
$(\delta_0,\delta_1,\delta_2,\delta_2+1,\delta_2+2,\ldots)$.
 \end{prop}
 \begin{proof} Without loss of generality, we assume that $\delta_0=0$.
 Let $M$ be a pure $R_1$-module of type $(0,\delta_1,\delta_2,\dots)$, with $\beta_1^{R_1}(M)=n$, $A$ be an $m \times n$ minimal presentation matrix of $M$, and $y\in \m_2$ be a linear $R_2$-regular element. 

 We construct a matrix $B=[b_{ij}]$ from $A$ with entries in $\m$ as follows:
 
 (i) If $m \geq n$, then $ b_{ij}=\left\{\begin{array}{lc}
     a_{ij} & \text{ if } i\neq j \\
      a_{ij}+y^{\delta_1} &  \text{ if } i= j
 \end{array}\right..$
 
 (ii) If $m<n$, then let $B= A' + y^{\delta_1}I_{n}$, where $A'$ is the $n \times n$ matrix obtained from $A$ by adding zero rows at the bottom.

 Note that in either case, $B$ has at least as many rows as columns. 
By the purity of $M$, and the construction of $B$, all the nonzero entries of $B$ are homogeneous of degree $\delta_1$. Let $F$ and $G$ be free $R$-modules generated in degrees $0$ and $\delta_1$ respectively, be such that $\phi: G \to F$ is the natural map induced by $B$. 
We claim that $N = \coker(\phi)$ is pure of type $(0,\delta_1,\delta_2,\delta_2+1,\delta_2+2,\ldots)$.
 
Clearly $N$ is generated in degree $0$, and $\Omega_1^R(N) \simeq \ker(\phi) \subset \m F$ is the $R$-submodule generated by columns of $B$.
By Proposotion \ref{first_syzygy}, recall that $$ \Omega_2^R(N)\simeq \ker(\phi) \simeq \left(\bigoplus_{i=1}^r\m_2({d_i})\right)\oplus\Omega_1^{R_1}(\pi_1(\ker(\phi)) \oplus  \Omega_1^{R_2}(\pi_2(\ker(\phi)) \oplus \left(\bigoplus_{j=1}^s\m_1({d_j'})\right). $$
By construction of $B$, observe that $r=s=0$, $\pi_1(\ker(\phi))\simeq \Omega_1^{R_1}(M)$, and $\pi_2(\ker(\phi))= y^{\delta_1}R_2^{\oplus n}$, which is a free $R_2$-module. Hence $\Omega_2^R(N)\simeq \Omega_2^{R_1}(M)$, which is generated in degree $\delta_2$, and has a linear resolution over $R_1$. Since $R_1, R_2$ are Koszul, by Proposition \ref{Prop-prep-for-koszul}, we get that $\Omega_2^R(N)$ has a linear resolution over $R$, which proves the result. 
 \end{proof}

\begin{thm}{\label{second_syzygy_linear}}
 Let $R$ be a fibre product of standard graded $\sk$-algebras. If $M$ is a pure $R$-module of infinite projective dimension, then the following holds:
 \begin{enumerate}[\rm a)]
     \item $\Omega_2^R(M)$ has linear resolution.
     \item $R$ is Koszul.
 \end{enumerate}
 In particular, if $\BB_{\mathbb{Q}}^{\text{pure}}(R) = \BB_{\mathbb{Q}}(R)$, then $R$ is Koszul.
 \end{thm}
 \begin{proof} Suppose $M$ has a pure resolution of type $(\delta_0,\delta_1,\delta_2,\ldots)$. Then $\Omega^R_j(M)$ is generated in degree $\delta_j$ for all $j$. In order to prove (i), we show by induction that $\delta_j=\delta_2+j-2$ for $j\geq 2$. This claim holds trivially for $j=2$. Suppose $\delta_j=\delta_2+j-2$ for some $j$.
 By Remark \ref{directSummandRemark},  there exists $i\in \{1,2\}$ such that $\m_i\mid \Omega_{j+1}^R(M)$, and hence $\Omega_{j+1}^R(M)$ has a generator of degree $\delta_j+1$. Thus, by purity of $M$, we see that $\Omega_{j+1}^R(M)$ is generated in degree $\delta_j+1=\delta_2+j-1$, proving (a).
 
 Since $\Omega_2^R(M)$ has a linear resolution, we have $\reg_R(\Omega_2^R(M))<\infty$. Hence, by using Proposition \ref{prop:koszulChar}, we get that $R$ is Koszul.
 
 Finally, observe that since $R$ is not regular, there exists a module of infinite projective dimension.  This implies that there exists a module $M$, with $\pdim_R(M)= \infty$, such that $\beta^R(M)$ spans an extremal ray of $\BB_{\mathbb{Q}}(R)$. If $\BB_{\mathbb{Q}}^{\text{pure}}(R)=\BB_{\mathbb{Q}}(R)$, then $M$ is pure, and hence by (b), $R$ is Koszul.
 \end{proof}

In \cite{Ku17}, it is proved that for any given standard graded $\sf k$-algebra $R$, if $\BB_{\mathbb{Q}}^{\text{pure}}(R)=\BB_{\mathbb{Q}}(R)$, then $R$ is Koszul, by using completely different techniques. The proof of Koszulness in Theorem \ref{second_syzygy_linear} above is easier, but specific to fibre products.
 \begin{rmk}\label{regularityRemark}{\rm
 Let $R$ be a fibre product of $\sk$-algebras and $M$ a pure $R$-module. 
 \begin{enumerate}[{\rm a)}]
     \item If $M$ is a pure $R$-module, since $\Omega_2^R(M)$ has a linear resolution, we see that the regularity of $M$ can be determined from the first three columns of Betti table of $M$, i.e., $$\reg_R(M)=\max\{j-i:\beta_{i,j}(M)\neq 0 \text{ for some } j, \text{ and } i\leq 2\}.$$
     \item The conclusion of {\rm (a)} holds for all $R$-modules $M$ such that $\beta^R(M) \in \BB_{\mathbb Q}^{\text{pure}}(R)$. As a consequence, if $\BB_{\mathbb Q}^{\text{pure}}(R) = \BB_{\mathbb Q}(R)$, then the conclusion of {\rm (a)} holds for all $R$-modules. 
\end{enumerate}
} \end{rmk}

If $M$ is a pure $R_1$-module of type $\delta$, then the following proposition gives a necessary condition on the $\delta$ for Betti table of $M$ over $R$ to be in the pure cone $\mathbb B _{\mathbb Q}^{\text{pure}}(R)$.
 \begin{prop}\label{DecompositionLemma}
 Let $R$ be a Koszul algebra, and $M$ be a pure $R_1$-module of type $\delta=(\delta_0,\delta_1,\delta_2,\delta_3,\dots)$.
 If $\beta^R(M)\in\BB_{\mathbb Q}^{\text{pure}}(R)$, then $\Omega_2^{R_1}(M)$ has a linear resolution, i.e.,
 for $i \geq 3$, $\delta_i=\infty$ or $\delta_{i-1} = \delta_i-1$.
 \end{prop}
 \begin{proof}
 Since $\beta^R(M)\in\BB_{\mathbb Q}^\text{pure}(R)$, there exist pure $R$-modules $M_1, \dots, M_r$ and $c_1,\dots, c_r\in \BQ_{>0}$ such that  $\beta^R(M)=\sum_j c_j\beta^R(M_j)$. 
Fix $i\geq 3$. If $\delta_i =\infty$, then we are done. If not, then we show that $\delta_{i-1} = \delta_i-1$.

 Since $\beta^{R_1}_{i,\delta_i}(M) \neq 0$, by Proposition \ref{FieldSyzygyLemma}, we get  $\beta_{i,\delta_i}^R(M) \neq 0$, and hence $\beta^R_{i,\delta_i}(M_j)\neq 0$ for some $j$. Note that by Theorem \ref{second_syzygy_linear}, $\Omega_2^R(M_j)$ has a linear resolution. Therefore, $\beta_{i-1,\delta_i-1}^R(M_j) \neq 0$, and hence $\beta_{i-1,\delta_i-1}^R(M) \neq 0$. Thus, by Proposition \ref{prop:firstNonzeroEntries}(b), 
 we have $\delta_i-1\leq \delta_{i-1}$. Since $\delta_{i-1}\leq \delta_i-1$, we are done.
 \end{proof}

Eisenbud-Schreyer \cite{ES09} proved that $ \BB_{\mathbb Q}(R)=\BB_{\mathbb Q}^\text{pure}(R)$ if $R$ is a polynomial ring. We observe that this is false in general for fibre product of polynomial rings. This follows from the next result.

\begin{cor}\label{cor:polyFibreProduct} Let $R=R_1\times_{\mathsf k} R_2$ with $\depth(R_2)\geq 3$. Then $\BB_{\mathbb Q}^{\text{pure}}(R) \neq \BB_{\mathbb Q}(R)$. 
 \end{cor}
 \begin{proof}
By Theorem \ref{second_syzygy_linear}, if $R$ is not Koszul, then $\BB_{\mathbb Q}^{\text{pure}}(R) \neq \BB_{\mathbb Q}(R)$. Hence, assume that $R$ is Koszul. 

Let $x_1,x_2,x_3 \in \m_2$ be a linear $R_2$-regular sequence, and $M=R/\langle x_1^2, x_2^2, x_3^2\rangle$.   
Then $\Omega_2^{R_2}(M)$ does not have a linear resolution, and hence, by  Proposition \ref{DecompositionLemma}, we get 
  $\beta^R(M)\notin \BB_{\mathbb Q}^\text{pure}(M).$ 
\end{proof}

 Gibbons \cite{Gi13} showed that if $R_1=\sk[X,Y]/\langle X^2,Y^2\rangle$, then $\BB_{\mathbb Q}(R_1)=\BB_{\mathbb Q}^\text{pure}(R_1)$. However, given any $R_2$, if $R=R_1 \times_{\mathsf{k}} R_2$, then  $\BB_{\mathbb Q}(R)\neq \BB_{\mathbb Q}^\text{pure}(R)$. This is captured in the following proposition.  

\begin{prop}\label{prop:gorensteinArtinian}
Let $R_1$ be a Gorenstein Artin standard graded $\sk$-algebra such that $\soc(R_1)$ is generated in degree $n\geq 2$. If $R= R_1 \times_{\mathsf{k}}R_2$, then $\BB_{\mathbb Q}(R)\neq \BB_{\mathbb Q}^\text{pure}(R)$. 
\end{prop}
\begin{proof}
As in the proof of the previous corollary, we may assume that $R$ is Koszul. Therefore, $R_1$ is Koszul by Corollary \ref{KoszulCorollary}, and hence $\sk$ has an $R_1$-linear resolution of the form
\[ \mathbb F_{\bullet} : \cdots 
\to R_1^{\beta_2}(-2) \xrightarrow[]{A} R_1^{\beta_1}(-1) \to R_1 \to \sk \to 0.\]
Let $A$ be a presentation matrix of $\Omega_1^{R_1}(\sk)$. Let $N= \coker(A^t)$, where $A^t$ denotes the transpose of $A$. Since $R_1$ is Gorenstein Artin, $\Hom(\_\_, R_1)$ is exact. Hence, we get the following exact complex
\[        0 \to \Hom_{R_1}(\sk, R_1)(-2) \simeq \soc(R_1)(-2)            \to R_1(-2)       \to R_1^{\beta_1}(-1)        \xrightarrow[]{A^t}  R_1^{\beta_2} \to  N \to  0.\]
Note that $N$ is a pure $R_1$-module, since $\soc(R_1) \simeq \sk(-n)$ and $R_1$ is Koszul. However, since $n \geq 2$, we see that $\Omega_2^{R_1}(N)$ does not have a linear resolution. Hence, by Proposition \ref{DecompositionLemma}, $\beta^R(N) \not \in \BB_{\mathbb{Q}}^{\text{pure}}(R)$, proving $\BB_{\mathbb Q}(R)\neq \BB_{\mathbb Q}^\text{pure}(R)$. 
\end{proof}

 Corollary \ref{cor:polyFibreProduct} and Proposition \ref{prop:gorensteinArtinian} were immediate consequences of Proposition \ref{DecompositionLemma} in which we saw examples where $\BB_{\mathbb Q}^{\text{pure}}(R) \neq \BB_{\mathbb Q}(R)$. In fact, we prove a more general result in Theorem \ref{MaintheoremFibreProducts}. Before that, we prove the following technical lemma.

\begin{lemma}\label{dimension} Let $R= R_1\times_{\mathsf{k}} R_2$ be a Koszul $\sk$-algebra with $\depth(R)=1$. Let $x_j$ be a linear $R_j$-regular element for $j=1,2$,  $J=\langle x_1, x_2\rangle$, and $M$ be an $R$-module generated in degree zero. Then 
\begin{enumerate}[{\rm a)}]
\item $J$ has a linear resolution.

\item If $M$ has a linear resolution with  $\beta_{i,i}^R(M)=c\beta_{i,i}^R(R/J)$ for all $i\geq 2$ for some $c\in \mathbb{Q}_{>0}$, then $\codim(M)\leq 1$.
\end{enumerate} 
\end{lemma}
 \begin{proof}
a) Note that $J= \langle x_1 \rangle \oplus \langle x_2 \rangle$, where $\langle x_1 \rangle$ is an $R_1$-module and $\langle x_2 \rangle$ is an $R_2$-module. Since $x_j$ is $R_j$-regular, by Lemma \ref{first_syzygy_for_direct_sum}, we get $\Omega_1^R(J)= \m_1(-1) \oplus \m_2(-1)$. Since $R$ is Koszul, by Corollary \ref{KoszulCorollary},  $R_1$ and $R_2$ are Koszul, and hence $\reg_{R_j}(\m_j) =1$ for $j=1,2$. Hence, by Proposition \ref{Prop-equality-of-reg}, we get that $J$ has a linear resolution. 
 
b) The conditions $\beta_{i,i}^R(M)=c\beta_{i,i}^R(R/J)$ for all $i\geq 2$ give us equalities at the level of Betti tables, and hence Hilbert series of $\Omega_2^R(M)$ and $\Omega_2^R(R/J)$. In particular, we have $\beta^R(\Omega_2^R(M))=c\beta^R(\Omega_2^R(R/J))$, and hence $H_{\Omega_2^R(M)}(z)=cH_{\Omega_2^R(R/J)}(z)$.

Now, $\Omega_2^R(R/J)= \m_1 (-1) \oplus \m_2(-1)$.  Therefore, $\m=\m_1 \oplus\m_2$ implies
\[H_{\Omega_2^R(R/J)}(z)=z H_{\mathfrak m_1}(z) + z H_{\mathfrak m_2}(z) =  z H_{\mathfrak m}(z)= z (H_R(z) -1),\] 
which gives $H_{\Omega_2^R(M)}(z)=cz(H_R(z)-1)$. 

Now, with $\beta_i = \beta_i^R(M)$, the exact sequence $0\longrightarrow \Omega_2^R(M) \longrightarrow R(-1)^{\beta_1}\longrightarrow R^{\beta_0}\longrightarrow M\longrightarrow 0$ gives 

	\[ H_M(z)= (\beta_0-z\beta_1)H_R(z)+H_{\Omega_2^R(M)}(z)=(\beta_0-z\beta_1)H_R(z)+cz(H_R(z)-1).
		\]	
  Let $m=\dim(M)$ and $d=\dim(R)$. We want to show that $d-m \leq 1$. 
  
Writing $H_M(z) = g(z)/(1-z)^m$, $H_R(z)= f(z)/(1-z)^d$ with $g(1)> 0$ and $f(1)>0$, and multiplying by  $(1-z)^m$, the above equality can be rewritten as
	\begin{eqnarray*}
		g(z)+cz(1-z)^m&=&\dfrac{(\beta_0-z\beta_1+cz)f(z)}{(1-z)^{d-m}}.
	\end{eqnarray*}	
Since $g(1)>0$, $c>0$, and $f(1)>0$ we get that $(1-z)^{d-m} \mid \beta_0-z(\beta_1-c)$. Therefore, $d-m\leq 1$, i.e., $\codim(M)\leq 1$.
\end{proof}

\begin{thm}\label{MaintheoremFibreProducts}
 Let $R_1, R_2$ be standard graded $\sk$-algebras, and $R= R_1 \times_{\mathsf k} R_2$  be such that $\depth(R)=1$. If  $\mathbb{B}_{\mathbb{Q}}(R)=\mathbb{B}^{\text{pure}}_{\mathbb{Q}}(R)$, then $R$ is Cohen-Macaulay and $H_R(z)= {(1+nz)}/{(1-z)}$ for some $n \in \mathbb{Z}$.
\end{thm}
\begin{proof} We first prove that $\dim(R)=1$.
Since $\depth(R)=1$, by Remark \ref{rmk:basicProperitesOfFibreProducts}(c), each $R_j$ has a linear regular element, say $x_j$. 
Let $ J=\langle x_1, x_2 \rangle R$ and $I= J+ \m^2$. Then, by Remark \ref{lemma-generators-up-to-degree-j}, we have $\beta_{i,i}^R(R/J)=\beta_{i,i}^R(R/I)$ for all $i \geq 2$.

 Suppose $\beta^R(R/I) =c_0 \beta^R(M_0) + \sum_{j=1}^r c_j \beta^R(M_j)$, where $M_0$ has a linear resolution and for $j\geq 1$,  $M_j$ has a pure resolution which is not linear. 
By Theorem \ref{second_syzygy_linear}, $\Omega_2^R(M_j)$ has a linear resolution, and hence $\beta_{i,i}^R(M_j)=0$ for all $i \geq 2$ and $j\geq 1$. In particular, we have $c_0\beta^R_{i,i}(M_0)=\beta^R_{i,i}(R/I)=\beta^R_{i,i}(R/J)$ for $i\geq 2$.  Since $\beta_{2,2}^R(R/J) \neq 0$, we have $c_0 >0$, and hence Lemma \ref{dimension} forces $\codim(M_0)\leq 1$.

The equality $\beta^R(R/I) = \sum_{j=0}^r c_j \beta^R(M_j)$ implies 
$H_{R/I}(z) = \sum_{j=0}^r c_j H_{M_j}(z)$. We may assume that $c_j >0$ for all $j$, and write $H_{M_j}(z) = g_j(z) / (1-z)^{d_j}$, where $g_j(1) >0$, and $d_j= \dim(M_j)$ for all $j \geq 0$. Moreover, $H_{R/I}(z) = a+bz$ for some $a, b \in \BZ$.

We now claim that the following hold for each $j$: (i) $M_j$ is Artinian, i.e., $d_j=0$, and (ii) $\deg(g_j) \leq 1$. 

Let $d= \max\{ d_j \mid 0 \leq j \leq r\}$, and $S= \{ j \mid \dim(M_j)=d\}$. Claim (i) follows if we show $d=0$. 

Suppose $d>0$. Consider the equation $H_{R/I}(z)=a+bz= \sum\limits_{j=0}^{r} c_j \dfrac{g_j(z)}{(1-z)^{d_j}}$. Multiplying both sides by $(1-z)^d$ and substituting $z=1$, we get $0= \sum_{j\in S}c_jg_j(1)$, which is a contradiction since the right hand side is positive. This proves Claim (i).

Since each $M_j$ is Artinian, the equality $H_{M_j}(z)=g_j(z)$ shows that the coefficients of $g_j$ are non-negative. In particular, $a+bz=\sum_{j=0}^r c_jg_j(z)$ forces $\deg(g_j) \leq 1$ for all $j$, proving Claim (ii).

By Claim (i), since $M_0$ is Artinian, we see that $\codim(M_0) \leq 1$ forces $\dim(R) \leq 1$. This proves that $R$ is Cohen-Macaulay, since $\depth(R)=1$.

Finally, writing $H_R(z)= f(z)/(1-z)$ with $f(1)>0$, the proof of Lemma \ref{dimension} applied to $M=M_0$ shows that $$g_0(z)+c_0z = \dfrac{(\beta_0 - \beta_1 z +c_0z) f(z)}{1-z}.$$
Multiplying both sides by $(1-z)$, and using the fact that $(1-z) \nmid f(z)$, we see that $\deg(f(z)) \leq 1$ by comparing degrees.
\end{proof}

In particular, Remark \ref{dim0and1extremals}, and Theorem \ref{MaintheoremFibreProducts} imply the following characterization.
 
\begin{thm}\label{thm:mainThmfibreProducts} Let $R_1, R_2$ be standard graded $\sk$-algebras of depth at least 1, and $R= R_1 \times_{\mathsf{k}} R_2$. Then the following are equivalent.
\begin{enumerate}[{\rm i)}]
    \item $\BB_{\mathbb{Q}}(R)= \BB^{\text{pure}}_{\mathbb{Q}}(R)$.
    \item $R$ is Cohen-Macaulay and $H_R(z)=(1+nz)/(1-z)$.
\end{enumerate}
In this case, the extremal rays of $\BB_{\mathbb{Q}}(R)$ are as observed in Remark \ref{dim0and1extremals}(b).
\end{thm}

\begin{rmk}\label{rmk:mainResults}{\rm Let $R$ be a standard graded $\sk$-algebra with $\BB_{\mathbb{Q}}(R) =\BB^{\text{pure}}_{\mathbb{Q}}(R)$.
 In \cite[Theorem 7.3.4]{Ku17}, it is proved that if $\depth(R)=0$, then $R$ is Artinian and level.
     Also observe that if $R$ is a fibre product with $\depth(R)=1$, then by Theorem \ref{MaintheoremFibreProducts}, $R$ is Cohen-Macaulay and level. 
    Thus, we have the following observations: 
\begin{enumerate}[a)]
    \item If $R$ is a fibre product with $\BB_{\mathbb{Q}}(R)= \BB^{\text{pure}}_{\mathbb{Q}}(R)$, then $R$ is Cohen-Macaulay and level.
    \item In the statement of Theorem \ref{thm:mainThmfibreProducts}, the condition on depth can be changed to $\dim(R_i) \geq 1$.
    \end{enumerate}  
In light of (b) above, we end this article with the following question.
}\end{rmk}

\begin{question}
Let $R= R_1 \times_{\mathsf k} R_2$ be Artinian (with $R_1 \neq \sk \neq R_2$).\\ If $\BB_{\mathbb{Q}}(R)= \BB^{\text{pure}}_{\mathbb{Q}}(R)$, then is $H_R(z)=1+nz$ for some $n \in \mathbb N$?
\end{question}


\begin{thebibliography}{10}

\bibitem{HAthesis}
H.~Ananthnarayan.
\newblock {\em Approximating {A}rtinian rings by {G}orenstein rings and 3-standardness of the maximal ideal}.
\newblock ProQuest LLC, Ann Arbor, MI, 2009.
\newblock Thesis (Ph.D.)--University of Kansas.

\bibitem{AK20}
H.~Ananthnarayan and R.~Kumar.
\newblock Extremal rays of {B}etti cones.
\newblock {\em J. Algebra Appl.}, 19(2):2050027, 19, 2020.

\bibitem{AP01}
L.~L. Avramov and I.~Peeva.
\newblock Finite regularity and {K}oszul algebras.
\newblock {\em Amer. J. Math.}, 123(2):275--281, 2001.

\bibitem{BF85}
J.~Backelin and R.~Fr\"{o}berg.
\newblock Koszul algebras, {V}eronese subrings and rings with linear resolutions.
\newblock {\em Rev. Roumaine Math. Pures Appl.}, 30(2):85--97, 1985.

\bibitem{BBEG12}
C.~Berkesch, J.~Burke, D.~Erman, and C.~Gibbons.
\newblock The cone of {B}etti diagrams over a hypersurface ring of low embedding dimension.
\newblock {\em J. Pure Appl. Algebra}, 216(10):2256--2268, 2012.

\bibitem{BS08}
M.~Boij and J.~S\"{o}derberg.
\newblock Graded {B}etti numbers of {C}ohen-{M}acaulay modules and the multiplicity conjecture.
\newblock {\em J. Lond. Math. Soc. (2)}, 78(1):85--106, 2008.

\bibitem{BH93}
W.~Bruns and J.~Herzog.
\newblock {\em Cohen-{M}acaulay rings}, volume~39 of {\em Cambridge Studies in Advanced Mathematics}.
\newblock Cambridge University Press, Cambridge, 1993.

\bibitem{AC2000}
A.~Conca.
\newblock Universally {K}oszul algebras.
\newblock {\em Math. Ann.}, 317(2):329--346, 2000.

\bibitem{DS75}
A.~Dress and H.~Kr\"{a}mer.
\newblock Bettireihen von {F}aserprodukten lokaler {R}inge.
\newblock {\em Math. Ann.}, 215:79--82, 1975.

\bibitem{ES09}
D.~Eisenbud and F.-O. Schreyer.
\newblock Betti numbers of graded modules and cohomology of vector bundles.
\newblock {\em J. Amer. Math. Soc.}, 22(3):859--888, 2009.

\bibitem{Fl12}
G.~Fl{\o}ystad.
\newblock Boij-{S}\"{o}derberg theory: introduction and survey.
\newblock In {\em Progress in commutative algebra 1}, pages 1--54. de Gruyter, Berlin, 2012.

\bibitem{GS16}
I.~Gheorghita and S.~V. Sam.
\newblock The cone of {B}etti tables over three non-collinear points in the plane.
\newblock {\em J. Commut. Algebra}, 8(4):537--548, 2016.

\bibitem{Gi13}
C.~Gibbons.
\newblock {\em Decompositions of {B}etti diagrams}.
\newblock ProQuest LLC, Ann Arbor, MI, 2013.
\newblock Thesis (Ph.D.)--The University of Nebraska - Lincoln.

\bibitem{Ku17}
R.~Kumar.
\newblock {\em Betti tables over standard graded rings}.
\newblock 2017.
\newblock Thesis (Ph.D.)--Indian Institute of Technology Bombay.

\end{thebibliography}
\end{document}